\documentclass[12pt]{amsart}
\usepackage{geometry}
\usepackage{graphicx}
\usepackage{amscd}
\usepackage{amssymb}
\usepackage{amsmath}
\usepackage[latin1]{inputenc}
\usepackage{epsfig,pstricks,pst-node,pst-tree,ifthen}
\usepackage{amsfonts}
\usepackage[all]{xy}

\theoremstyle{plain}                                       %
\newtheorem{thm}{\quad Theorem}                            %
\newtheorem{cor}[thm]{\quad Corollary}                     %
\newtheorem{prop}[thm]{\quad Proposition}                  %
\theoremstyle{definition}                                  %
\newtheorem{rmk}[thm]{\quad Remark}                        %
\newtheorem{ejem}[thm]{\quad Example}                      %

\newcommand{\R}{{\Bbb R}}
\newcommand{\N}{{\Bbb N}}
\newcommand{\K}{{\Bbb K}}
\newcommand{\C}{{\Bbb C}}

\newcommand{\Z}{{\Bbb Z}}

\newcommand{\al}{{\alpha}}
\newcommand{\be}{{\beta}}

\newcommand{\om}{{\omega}}

\newcommand{\g}{{\gamma}}
\newcommand{\de}{{\delta}}

\newcommand{\Ri}{{\mathcal{R}}}

\begin{document}

\vspace{1cm}

\title{The group generated by Riordan involutions.}

\author{Ana Luz\'{o}n*, Manuel A. Mor\'{o}n$^\natural$ and L. Felipe Prieto-Martinez$\dag$ }
\address{*Departamento de Matemática Aplicada. Universidad Polit\'{e}cnica de Madrid (Spain).}
\email{anamaria.luzon@upm.es}

\address{ $\natural$ Departamento de Algebra, Geometría y Topología.
Universidad Complutense de Madrid  and Instituto de Matemática
Interdisciplinar (IMI)(Spain).} \email{mamoron@mat.ucm.es}

\address{ $\dag$ IES Alpajés. Aranjuez.(Spain).}
\email{felipe.prieto@educa.madrid.org}

\maketitle

\begin{abstract}
We prove that any element in the group generated by the Riordan
involutions is the product of at most four of them. We also give a
description of this subgroup as a semidirect product of a special
subgroup of the commutator subgroup and the Klein four-group.
\end{abstract}

Keywords:  Riordan groups, product of involutions, commutator
subgroup, inverse limit.

 MSC:  20H20, 15B99.


\section{Introduction}

To set our results in a wider context, we have to say that the group
generated by the involutions in a group $G$ has been studied for
different kinds of groups. In many of these groups has been obtained
that any element in the group generated by the involutions in $G$ is
the product of at most four of them. See as a sample
\cite{Kasner,HalmosKakutani1958,Wonen,Djokovic,GHR,Knuppel91,Knuppel98,ofarrell1,Slowik1,ofarrel2,Malcolm}
and the references therein. The aim of this paper is to prove the
same result for the Riordan groups $\Ri(\K)$ and $\Ri_n(\K)$,
$n\geq1$. That is,

\begin{thm}\label{T:main}
Let $n\geq1$ and $G=\Ri(\K)$ or $\Ri_n(\K)$. Then, any element in
the group generated by the involutions in $G$ is the product of at
most four of them.
\end{thm}

The Riordan group $\Ri(\K)$  is a multiplicative group whose
elements are some special infinite lower triangular matrices with
entries in $\K$, a field of characteristic zero. Some of its
subgroups and many of its elements appear in many different context
along the time. For instance in Rota's and collaborators'
description of Umbral Calculus or equivalently the study of Sheffer
sequences of polynomials \cite{Rota2}, see also
\cite{Barnabei,KnuthConvo, He-H-S, WangWang, poly, selfShef}. Even
more, every generalized Appell polynomial sequence \cite{Boas-Buck}
can be obtained by means of a Riordan matrix. In particular, the
matrices associated to many classical polynomial sequences are
Riordan matrices. For example: Chebyshev polynomials, Fibonacci
polynomials, Pell polynomials, Morgan-Voyce polynomials, Fermat
polynomials. See \cite{poly}. Of course, the sequence of binomial
polynomials gives rise to the oldest and most studied Riordan
matrix: Pascal's triangle.

The group of invertible infinite lower triangular Toeplitz matrices
is a subgroup of the Riordan group. Also, the family of all
invertible Jabotinsky matrices. These matrices can be considered as
matrices associated to composition operators in suitably chosen
spaces. Riordan matrices correspond to invertible \textit{weighted
composition operators} in an appropriate space. See \cite{teo}. The
above observation describes the way a Riordan matrix transforms
power series. The formula representing this transformation is called
the Fundamental Theorem of Riordan Arrays by some authors in the
specific literature. The group of the formal power series of order
one with composition and the so-called substitution group of formal
power series are naturally isomorphic to some subgroups of the
Riordan group. See for example
\cite{Bennett,JabotinskyFrances,Jabotinsky53,Scheinberg,ofarrell1,Babenko}.
Verde-Star in \cite{VerdeStar85} describes a group of operators in
the multivariate context that now can be interpreted as an extension
of the Riordan group in several variables. See
\cite{CheonHuang,ofarrel3} for some recent related results.

The Riordan group was introduced, under this name and in a more
restrictive context, by L. Shapiro and collaborators in
\cite{Sha91}. Soon after R. Sprugnoli \cite{Spr94} obtained many
combinatorial identities using this group. Because of the origin of
Riordan arrays, they are intrinsically related to Combinatorics.
There, there are its first and most of its applications. Currently,
the Riordan group is being studied under different angles, including
the algebraic structure, the construction of new matrices from old,
polynomials associated, Riordan pattern quest, among some others. A
non-exhaustive list of works related to Riordan group is the
following \cite{BarryMeixner, Cheon-Jin,Cheon, He,
Can-kwant,Constr,Complem,Double,Merlini,Rog78,Sha4,WangWang,Yang18}
and references therein. The first and second authors ran into this
group from a fixed point problem to compute the quotient of series
$\displaystyle{\frac{f}{g}}$, see  \cite{teo,BanPas,2ways}. This
caused the $T(f\mid g)$ notation for a Riordan matrix.

In \cite{finitas} $\Ri(\K)$ is described as an inverse limit of a
certain inverse sequence of groups $\Ri_n(\K)$ formed by
$(n+1)\times (n+1)$ invertible lower triangular matrices. For $\K=
\R$ or $\C$, the groups $\Ri_n(\K)$ have a natural structure of Lie
groups. Using both facts above, in \cite{Lie} the authors got and
exploited a Frechet-Lie group structure in $\Ri(\K)$. Another
consequence of the inverse limit approach is the description of any
involutions in every Riordan group $\Ri_n(\K)$ and $\Ri(\K)$, see
\cite{formula}. This is the starting point for this paper.

In Section 2 we recall some basic facts needed about Riordan
matrices, specially those related to involutions in Riordan groups.
In Section 3 we observe that the subgroup of Riordan matrices with
ones in the main diagonal plays a significant role in our context.
In fact, we prove that this subgroup is the commutator subgroup and
that any of its elements is a commutator. We illustrate our result
with Pascal's triangle example. Section 4 contains the main results
in this paper. In particular, a detailed proof of Theorem
\ref{T:main}, a description of the group generated by involutions
involving the semidirect product of a special subgroup of the
commutator subgroup and a copy of the Klein four-group. Finally, we
compute the minimal number of involutions needed to describe any
element in the group generated by them in $\Ri_n$, $n\geq1$ and in
$\Ri$.

To end this introduction, we would like to note that Riordan groups,
far from being finite and/or simple, share with finite non-abelian
simple groups the commutator and involution width. See the recent
crucial papers \cite{Ore} and \cite{Malcolm}.


\section{Previous results.}
In this paper $\N$ represents the set $\{0,1,2,3,\cdots\}\subset\K$.
An element $D=(d_{i,j})_{i,j\in\N}$ in the Riordan group $\Ri(\K)$
denoted by $(d,h)$ or $T(f\mid g)$ is an infinite matrix whose
entries are $d_{i,j}=[x^i]d(x)h^j(x)$ for $(d,h)$ notation or
$\displaystyle{d_{i,j}=[x^i]\frac{x^jf(x)}{g^{j+1}(x)}}$ for
$T(f\mid g)$ notation, with $d,f,g\in\K[[x]]$ invertible for Cauchy
product and $h\in\K[[x]]$ invertible for composition operation.
Moreover, $[x^k]$ denotes the $k$-th coefficient in the series
expansion. Note that, by definition, these matrices are invertible
infinite lower triangular. In terms of the parameters the operations
in the group are:
\[
(d,h)(l,m)=(dl(h),m(h)),  \qquad
(d,h)^{-1}=\left(\frac{1}{d(h^{-1})},h^{-1}\right)
\]
where $h^{-1}\circ h=h\circ h^{-1}=x$. Or
\[
T(f\mid g)T(r\mid
s)=T\left(fr\left(\frac{x}{g}\right)\Big|gs\left(\frac{x}{g}\right)\right),
\qquad T^{-1}(f\mid g)=
T\left(\frac{1}{f\left(\frac{x}{A}\right)}\Big| A\right)
\]
where $\frac{x}{g}\circ\frac{x}{A}=\frac{x}{A}\circ\frac{x}{g}=x$
and $\al\beta\left(\frac{x}{\gamma}\right)$ means
$\al(x)\beta\left(\frac{x}{\gamma(x)}\right)$. The series $A$ is the
so-called $A$-sequence associated to the Riordan matrix $T(f\mid
g)$. The $A$-sequence allows us to construct horizontally, i. e. by
rows, in the following way
\[
d_{i,j}=\sum_{k=0}^{i-j}a_kd_{i-1,j-1+k} \qquad i,j\geq1
\]
where $A=\sum_{n\geq0}a_nx^n$. See \cite{Rog78}. Note that the
series $g=\sum_{n\geq0}g_nx^n$ allows us to construct vertically,
i.e. by columns, a Riordan matrix in the following way
\[
d_{i,j}=\sum_{k=0}^{i-j}g_kd_{i+1-k,j+1}  \qquad i,j\geq0.
\]
See \cite{teo,2ways}. In the special case that $D=T(f\mid g)$ is a
Riordan involution, then $A=g$. In fact, g is always the
$A$-sequence of the inverse of $D$. See Proposition 7 in
\cite{2ways}.

The action induced by $D=(d,h)$ in $\K[[x]]$ is given by
\[
(d,h)\al=d\al(h) \quad \text{for}\ \al\in\K[[x]].
\]
That is, $(d,h)$ is a weighted composition operator in $\K[[x]]$.
One can get consistently the corresponding formula for the $T(f\mid
g)$ notation.

For every $n\in\N$ consider the general linear group $GL(n+1,\K)$
formed by all $(n+1)\times(n+1)$ invertible matrices with
coefficients in $\K$. In the sequel, if it cause not confusion, we
denote only by $\Ri$ or $\Ri_n$ to refer to Riordan groups. Since
every Riordan matrix is lower triangular, we can define a natural
homomorphism $\Pi_n: \mathcal{R}\rightarrow GL(n+1,\K)$ given by $
\Pi_n((d_{i,j})_{i,j\in\N})=(d_{i,j})_{i,j=0,1,\cdots,n}. $ Consider
the subgroup $\mathcal{R}_n=\Pi_n(\mathcal{R})$ of $GL(n+1,\K)$. We
can recover the group $\mathcal{R}$ as the inverse limit of the
inverse sequence of groups
$\{(\mathcal{R}_n)_{n\in\N},(P_n)_{n\in\N}\}$ where
$P_n:\Ri_{n+1}\rightarrow\Ri_n$ is such that if $D\in\Ri_{n+1}$,
$P_n(D)$ is obtained from $D$ by deleting its last row and its last
column, i.e.
$P_{n}((d_{i,j})_{i,j=0,1,\cdots,n+1})=(d_{i,j})_{i,j=0,1,\cdots,n}$.
See \cite{finitas}. Obviously if $n=0$ then $\Ri_0=\K^*$ with the
usual product in $\K$ being $\K^{*}=\K\setminus\left\{0\right\}$.

Later in \cite{formula}, using the inverse limit approach described
above we got

\begin{thm}\label{T:invo}{(\bf Riordan Involution's Formula)}
Suppose $n\geq2$. Let $D=(d_{i,j})\in\Ri_{n-1}$ be an involution
and take $\hat{D}=(d_{i,j})\in\Ri_{n}$ such that
$P_{n-1}(\hat{D})=D$.

(a) If $n$ is even, $\hat{D}$ is an involution if and only if
$d_{n,1}$ is arbitrary and
\begin{equation}\label{E:d(n,0) n par}
d_{n,0}=-\frac{1}{2d_{0,0}}\sum_{k=1}^{n-1}d_{n,k}d_{k,0}
\end{equation}

(b) If $n$ is odd, $\hat{D}$ is an involution if and only if
$d_{n,0}$ is arbitrary and
\begin{equation}\label{E:d(n,1) n impar}
d_{n,1}=-\frac{1}{2d_{1,1}}\sum_{k=2}^{n-1}d_{n,k}d_{k,1}
\end{equation}

Moreover, if $a_0,\cdots,a_{n-2}$ are the parameters in Theorem 5 in
\cite{finitas} to construct $D$, then the needed $a_{n-1}$ to
construct $\hat{D}$ is given by the formula

\begin{equation}\label{e:a-seq}
a_{n-1}=\frac{1}{d_{n-1,n-1}}\left(d_{n,1}-\sum_{j=0}^{n-2}a_jd_{n-1,j}\right)
\end{equation}
\end{thm}

A consequence of Theorem \ref{T:invo} is the following.

\begin{cor} \label{C:alfabeta} Let $\al=\sum_{i\in\N}\al_i x^i$  be
an arbitrary formal power series then
\begin{itemize}
\item[(i)] There is an unique nontrivial involution
$D=(d_{i,j})_{i,j\in\N}$ such that
\[
d_{0,0}=1, \qquad d_{2i+1,0}=\al_{2i} \qquad \text{and} \qquad
d_{2i+2,1}=\al_{2i+1} \qquad \text{for} \qquad i=0,1,\cdots
\]
we denote it by  $\mathcal{I}^+_{\al}$.
\item[(ii)] There is an unique nontrivial involution
$D=(d_{i,j})_{i,j\in\N}$ such that
\[
d_{0,0}=-1, \qquad d_{2i+1,0}=\al_{2i} \qquad \text{and} \qquad
d_{2i+2,1}=\al_{2i+1} \qquad \text{for} \qquad i=0,1,\cdots
\]
we denote it by  $\mathcal{I}^-_{\al}$.
\end{itemize}
Moreover, any nontrivial Riordan involution can be constructed by
this way.
\end{cor}

The next two propositions were also obtained in \cite{formula}.
$T(f\mid g)$ notation seems to be specially adequate to obtain them.

\begin{prop}
If $T(f\mid g)$ is a Riordan involution then $g_2=0$, where
$g=\sum_{n\geq0}g_nx^n$.
\end{prop}

\begin{prop}\label{P:grupo g2=0}
If $\Omega_0=\{T(f\mid g)\in\Ri, \ \mid \ g_2=0\}$, then $\Omega_0$
is a subgroup of $\Ri$.
\end{prop}

\begin{rmk}\label{R:a2=0}
The group $\Omega_0$ above can be described as the set of Riordan
matrices whose $A$-sequences have null quadratic coefficient,
because if $T(f\mid g)$ is an involution and $A$ its $A$-sequence
then $g=A$. On the other hand, the condition $g_2=0$ turns into
\begin{equation} \label{E:a2=0 con hs}
h_2^2=h_1h_3
\end{equation}
in the $(d,h)$ notation, for $h=\sum_{n\geq1}h_nx^n$.

It is no difficult to prove that $\Omega_0$ is not normal in the
Riordan group.
\end{rmk}

\section{The commutator subgroup and some relations with involutions.}


Let $ \mathfrak{I} $ be the set of all Riordan involutions and
denote by $<\mathfrak{I}>$ the group generated by $ \mathfrak{I} $.
To describe $<\mathfrak{I}>$ some observations are relevant. First,
note that every Riordan matrix $(d,h)\in\Ri$ can be written as the
product of a Riordan matrix with $1$'s in the main diagonal and a
diagonal Riordan matrix as follows
\[
(d,h)\in \Ri,
 \qquad (d,h)=\left(\frac{d}{d_0},\frac{h}{h_1}\right)(d_0,h_1x).
\]
Second, every product of Riordan involutions has in its main
diagonal 1's or -1's or alternatively 1 and -1 starting at 1 or at
-1. So, in the case that  $(d,h)\in <\mathfrak{I}> $ we get that
$(d_0,h_1x)\in \mathcal{K}$ where $\mathcal{K}=\{I, -I,
\mathcal{I}^{+}_0, \mathcal{I}^{-}_0 \}$. Moreover, $\mathcal{K}$ is
a subgroup of the Riordan group (the Klein four-group), $
\mathcal{K} \approx \Z_2\times\Z_2$. Third, from Proposition 20 and
Proposition 21 in \cite{formula} we obtain that
$<\mathfrak{I}>\leq\Omega_0$. Finally, the matrix
$\left(\frac{d}{d_0},\frac{h}{h_1}\right)$ has 1's in its main
diagonal and if $(d,h)\in <\mathfrak{I}> $ then
$\left(\frac{d}{d_0},\frac{h}{h_1}\right)\in\Omega_0$.

As a summary,
\begin{itemize}
\item[(i)] $\left(\frac{d}{d_0},\frac{h}{h_1}\right)$ has 1's in its main
diagonal.
\item[(ii)] If $(d,h)\in <\mathfrak{I}>$ then
$\left(\frac{d}{d_0},\frac{h}{h_1}\right)\in\Omega_0$ and
$(d_0,h_1x)\in \mathcal{K}$.
\item[(iii)] $<\mathfrak{I}>\unlhd \ \Omega_0$.
\end{itemize}

In view of the above observations, Riordan matrices with 1's in the
main diagonal play an important role. In fact,

\begin{thm}
The commutator subgroup of $\Ri$, denoted by $[\Ri,\Ri]$, is formed
by all Riordan matrices with 1's in the main diagonal. That is,
\[ [\Ri,\Ri]=\{(d,h)\in \Ri, \ / \ d_0=1, \ h_1=1\}.\]
Moreover, every element in  $[\Ri,\Ri]$ is a commutator.
\end{thm}
\begin{proof}
Consider the set
\[\mathrm{C}=\{(d,h)\in \Ri, \ / \ d_0=1, \ h_1=1\}.\]
Obviously, $\mathrm{C}$ is a subgroup of $\Ri$. If $D\in[\Ri,\Ri]\
\Rightarrow \exists \ C_1, C_2, \cdots, C_k$ commutators such that
$D=C_1C_2\cdots C_k.$ But all $C_i$'s are commutators and triangular
matrices then $C_i\in\mathrm{C}$, so $D\in \mathrm{C}$. Consequently
$[\Ri,\Ri]\subseteq \mathrm{C}$.

If $D\in\mathrm{C}$ then $D=(d,h)$ with $d_0=1$ and $h_1=1$. Let
$r\in\K$ such that $r\neq0, r^n\neq1, \forall n\geq1$. Consider the
diagonal Riordan matrix $A_r=(1,rx)$. We are going to prove that
there exists at least a Riordan matrix $B_r=(l,m)$ depending on $r$,
and of course on $D$, such that $D=[A_r,B_r]$. That is,
\[
(d,h)=(1,rx)(l,m)\left(1,\frac{x}{r}\right)\left(\frac{1}{l(m^{-1})},m^{-1}\right).\]
Using the product in the Riordan group we get
\[
(d,h)=\left(\frac{l(rx)}{l\left(m^{-1}\left(\frac{m(rx)}{r}\right)\right)},m^{-1}\left(\frac{m(rx)}{r}\right)\right)
\]
hence
\[
h=m^{-1}\left(\frac{m(rx)}{r}\right) \quad \text{and} \quad
d=\frac{l(rx)}{l\left(m^{-1}\left(\frac{m(rx)}{r}\right)\right)}
\]
\[
h=m^{-1}\left(\frac{m(rx)}{r}\right)\ \Leftrightarrow \
m(h)=\frac{m(rx)}{r} \ \Leftrightarrow \
(1,h)m=\left(\frac{1}{r},rx\right)m
\]
Consequently for $n\geq2$
\[
m_n=\frac{1}{r^{n-1}-1}\sum_{k=1}^{n-1}[x^n]h^km_k
\]
with $m_1\neq0$ arbitrary and $[x^n]h^k$ represents the coefficient
of $x^n$ of the $k$-th of the formal power series $h$. Analogously,
\[
d=\frac{l(rx)}{l\left(m^{-1}\left(\frac{m(rx)}{r}\right)\right)} \
\Leftrightarrow \ d=\frac{l(rx)}{l(h)} \ \Leftrightarrow \
\]
\[
\Leftrightarrow \ dl(h)=l(rx) \ \Leftrightarrow \ (d,h)l=(1,rx)l
\]
then $l_0\neq0$ and
\[
l_n=\frac{1}{r^n-1}\sum_{k=0}^{n-1}d_{n,k}l_k, \quad \text{for}\
n\geq1
\]
Hence $D$ is a commutator and of course $D\in[\Ri, \Ri]$.
\end{proof}

\begin{rmk}
Note that we have actually proved something stronger than the
statement of the theorem because we can use the same matrix $A_r$
for any commutator.
\end{rmk}

\begin{ejem} In view of the above result, Pascal's triangle
$\left(\frac{1}{1-x},\frac{x}{1-x}\right)$ is a commutator in the
Riordan group. For example for $r\neq0,1$ we get
\[
\left(\frac{1}{1-x},\frac{x}{1-x}\right)=(1,rx)
\left(\frac{r-1}{r-1-x},\frac{rx-x}{r-1-x}\right)\left(1,\frac{x}{r}\right)
\left(\frac{r-1}{r-1+x},\frac{rx-x}{r-1+x}\right).
\]

\end{ejem}

Since the abelianization of a group $G$ is the quotient group
$G/[G,G]$ it is clear that
\begin{cor}
The abelianization of the Riordan group is isomorphic to the
subgroup formed by all diagonal Riordan matrices. Consequently, it
is isomorphic to the direct product $\K^{*}\times\K^{*}$.
\end{cor}

Note also that, in this case:
\[
\Ri\approx[\Ri,\Ri]\rtimes\Ri/[\Ri,\Ri]
\]

\begin{rmk}
For $\Ri_n$ (Riordan groups of finite matrices) analogous results
hold. $[\Ri_n,\Ri_n]$ is formed by all $(n+1)\times(n+1)$ Riordan
matrices with 1's in the main diagonal. Moreover, for $n\geq1$ the
abelianization of $\Ri_n$ is also isomorphic to
$\K^{*}\times\K^{*}$. Note that the group $\Ri_0$ is
$(\K^{\star},\cdot)$ and then it is abelian. In the remaining cases
we obtain
\[
\Ri_n\approx[\Ri_n,\Ri_n]\rtimes\Ri_n/[\Ri_n,\Ri_n]
\]
\end{rmk}

\section{The group generated by Riordan involutions}

To prove Theorem \ref{T:main} we have to use the following partial
fact that describes how certain Riordan matrices can be expressed as
product of three Riordan involutions. We are going to construct the
three involutions doing an exhaustive use of the Riordan
involution's Formula in Theorem \ref{T:invo}. Theorem \ref{T:3invo}
below is the key to obtain the general result.

\begin{thm}\label{T:3invo}
Let $d=\sum_{n\geq0}d_nx^n$ and $h=\sum_{n\geq1}h_nx^n$  be two
power series such that $d_0=1$ and $h_1=-1$. Suppose also that the
Riordan matrix $(d,h)\in\Omega_0$, then there are three Riordan
involutions, $\mathcal{I}^{+}_{\al}=(\de_1,\om_1), \
\mathcal{I}^{+}_{\be}=(\de_2,\om_2), \
\mathcal{I}^{+}_{\g}=(\de_3,\om_3)$ such that
$(d,h)=(\de_1,\om_1)(\de_2,\om_2)(\de_3,\om_3)$
\end{thm}

\begin{proof}
To prove this theorem is equivalent to show that the system of
functional equations
\begin{equation}\label{E:ladproductode3}
d(x)=\de_1(x)\de_2(\om_1(x))\de_3(\om_2(\om_1(x)))
\end{equation}
\begin{equation}
\label{E:lahproductode3}
    h(x)=\om_3(\om_2(\om_1(x)))
\end{equation}
has solutions with $\mathcal{I}^{+}_{\al}=(\de_1,\om_1), \
\mathcal{I}^{+}_{\be}=(\de_2,\om_2)$ and $\
\mathcal{I}^{+}_{\g}=(\de_3,\om_3)$ Riordan involutions. We begin by
equation (\ref{E:lahproductode3}). Notice that the matrices
$(1,\om_1), \ (1,\om_2), \ (1,\om_3)$ are also involutions. Equation
(\ref{E:lahproductode3}) is equivalent to
\[
(1,\om_1)(1,h)=(1,\om_2)(1,\om_3) \quad \text{and then to} \quad
(1,\om_1)h=(1,\om_2)\om_3
\]

Suppose now
\[
(1,\om_1)=(a_{i,j})_{i,j\in\N}, \quad
(1,\om_2)=(b_{i,j})_{i,j\in\N}, \quad (1,\om_3)=(c_{i,j})_{i,j\in\N}
\]

The equation in $\Ri_2$ is
\[
\left(
  \begin{array}{ccc}
    1 & 0 & 0  \\
    0 & -1 & 0  \\
    0 & a_{2,1} & 1  \\
    \end{array}
\right)\left(
         \begin{array}{c}
           0 \\
           -1 \\
           h_2\\
          \end{array}
       \right)=
       \left(
  \begin{array}{ccc}
    1 & 0 & 0 \\
    0 & -1 & 0  \\
    0 & b_{2,1} & 1  \\
  \end{array}
\right)\left(
         \begin{array}{c}
           0 \\
           -1 \\
           c_{2,1}\\
         \end{array}
       \right)
\]
The above equality is equivalent to the linear equation
\begin{equation}\label{E:n=2}
 a_{2,1}-b_{2,1}+c_{2,1}=h_2
\end{equation}
that has infinite solutions because $a_{2,1},b_{2,1}$ and $c_{2,1}$
can be choosen arbitrarily by (a) in Theorem \ref{T:invo}.

The equation in $\Ri_3$ is
\[
\left(
  \begin{array}{cccc}
    1 & 0 & 0 & 0 \\
    0 & -1 & 0 & 0 \\
    0 & a_{2,1} & 1 & 0 \\
    0 & -a_{2,1}^2 & -2a_{2,1} & -1 \\
  \end{array}
\right)\left(
         \begin{array}{c}
           0 \\
           -1 \\
           h_2\\
           h_3 \\
         \end{array}
       \right)=
       \left(
  \begin{array}{cccc}
    1 & 0 & 0 & 0 \\
    0 & -1 & 0 & 0 \\
    0 & b_{2,1} & 1 & 0 \\
    0 & -b_{2,1}^2 & -2b_{2,1} & -1 \\
  \end{array}
\right)\left(
         \begin{array}{c}
           0 \\
           -1 \\
           c_{2,1}\\
           -c_{2,1}^2 \\
         \end{array}
       \right)
\]
The above equality is equivalent to the system
\[
\left\{
  \begin{array}{ll}
    a_{2,1}-h_2=b_{2,1}-c_{2,1} \\
    a_{2,1}^2-2h_2a_{2,1}-h_3=(b_{2,1}-c_{2,1})^2.
  \end{array}
\right.
\]
Since $(1,h)\in\Omega_0$,  using (\ref{E:a2=0 con hs}) the system
above reduces to the unique linear equation $(\ref{E:n=2})$ which is
just the same as in $\Ri_2$. Then it has solutions.

Suppose now that the equation in $\Ri_n$ has solution. In
$\Ri_{n+1}$ only the following new equation appears
\begin{equation}\label{E:ecu1}
\sum_{k=1}^{n+1}a_{n+1,k}h_k=\sum_{k=1}^{n+1}b_{n+1,k}c_{k,1}
\end{equation}

If $n$ is odd, (\ref{E:ecu1}) becomes
\[
-a_{n+1,1}+\sum_{k=2}^{n}a_{n+1,k}h_k+h_{n+1}=-b_{n+1,1}+
\sum_{k=2}^nb_{n+1,k}c_{k,1}+c_{n+1,1}
\]
that is
\begin{equation}\label{E:ecu2}
a_{n+1,1}-b_{n+1,1}+c_{n+1,1}=h_{n+1}+\sum_{k=2}^{n}(a_{n+1,k}h_k-b_{n+1,k}c_{k,1})
\end{equation}
Induction hypothesis allows us to set solutions in $\Ri_n$. Once we
take one of them, all in the right side of equation $(\ref{E:ecu2})$
is known by the construction  of  Riordan matrices by rows.
Moreover, by (a) in Theorem \ref{T:invo}, $a_{n+1,1}, \ b_{n+1,1}, \
c_{n+1,1}$ can be taken arbitrarily to construct the involutions.
So, $(\ref{E:ecu2})$ has infinite solutions. The system obtained
adding $(\ref{E:ecu2})$ to the previous one used to compute
solutions in $\Ri_n$ is consistent. Then we have solutions for
$n+1$.

If $n$ is even, the coefficients $a_{n+1,1}, \ b_{n+1,1}, \
c_{n+1,1}$ can not be taken arbitrarily to construct the
involutions. From (\ref{E:d(n,1) n impar}) in Theorem \ref{T:invo}
they depend, in particular, on $a_{n,1}, b_{n,1}$ and $c_{n,1}$. So,
to be sure of the existence of solutions in $\Ri_{n+1}$, assuming
that they exist in $\Ri_n$, we have to study the consistency of the
system
\[
\left\{
  \begin{array}{ll}
    a_{n,1}-b_{n,1}+c_{n,1}=h_{n}+\sum_{k=2}^{n-1}(a_{n,k}h_k-b_{n,k}c_{k,1}) \\
   \sum_{k=1}^{n+1}a_{n+1,k}h_k=\sum_{k=1}^{n+1}b_{n+1,k}c_{k,1}
  \end{array}
\right.
\]
where the unknown variables are $a_{n,1}, b_{n,1}$ and $c_{n,1}$.
Using (\ref{E:d(n,1) n impar}) in Theorem \ref{T:invo} we get
\[
a_{n+1,1}=\frac{1}{2}\sum_{k=2}^na_{n+1,k}a_{k,1}=
\frac{1}{2}(a_{n+1,2}a_{2,1}+a_{n+1,n}a_{n,1})+K_1
\]
\[
b_{n+1,1}=\frac{1}{2}\sum_{k=2}^nb_{n+1,k}b_{k,1}=
\frac{1}{2}(b_{n+1,2}b_{2,1}+b_{n+1,n}b_{n,1})+K_2
\]
\[
c_{n+1,1}=\frac{1}{2}\sum_{k=2}^nc_{n+1,k}c_{k,1}=
\frac{1}{2}(c_{n+1,2}c_{2,1}+c_{n+1,n}c_{n,1})+K_3
\]
where
\[
K_1=\frac{1}{2}\sum_{k=3}^{n-1}a_{n+1,k}a_{k,1}, \quad
K_2=\frac{1}{2}\sum_{k=3}^{n-1}b_{n+1,k}b_{k,1}, \quad
K_3=\frac{1}{2}\sum_{k=3}^{n-1}c_{n+1,k}c_{k,1}.
\]

Note that $K_1,\  K_2, \ K_3$ can be computed once one set a
solution in $\Ri_{n-1}$ by induction hypothesis and by the
construction by rows of Riordan matrices.

Since $n$ is even we get
\[
a_{n+1,2}=\sum_{k=0}^{n-1}A^{\om_1}_ka_{n,1+k}=
-a_{n,1}+\sum_{k=1}^{n-2}A^{\om_1}_ka_{n,1+k}+A_{n-1}^{\om_1}
\]
where  $A^{\om_i}_k$ is the $k$-th coefficient of the $A$-sequence,
$A^{\om_i}$, of the involution $(1,\om_i)$ for $i=1,2,3$. From
(\ref{e:a-seq}) in Theorem \ref{T:invo} and recalling that
$a_{n,0}=0$ for $n\geq1$, we obtain that
\[
A_{n-1}^{\om_1}=\frac{1}{a_{n-1,n-1}}\left(a_{n,1}-\sum_{k=0}^{n-2}A^{\om_1}_ka_{n-1,k}\right)=
-a_{n,1}+\sum_{k=1}^{n-2}A^{\om_1}_ka_{n-1,k}
\]
hence
\[
a_{n+1,2}=-2a_{n,1}+K_4,
\]
where
\[
K_4=\sum_{k=1}^{n-2}A^{\om_1}_k(a_{n,1+k}+a_{n-1,k}). \quad
\]
In a similar way we get
\[
 b_{n+1,2}=-2b_{n,1}+K_5, \quad
c_{n+1,2}=-2c_{n,1}+K_6,
\]
where
\[
K_5=\sum_{k=1}^{n-2}A^{\om_2}_k(b_{n,1+k}+b_{n-1,k}), \quad
K_6=\sum_{k=1}^{n-2}A^{\om_2}_k(c_{n,1+k}+c_{n-1,k}). \quad
\]
Now, we are going to write the equation
\[
\sum_{k=1}^{n+1}a_{n+1,k}h_k=\sum_{k=1}^{n+1}b_{n+1,k}c_{k,1}
\]
as
\[
-a_{n+1,1}+a_{n+1,2}h_2+\sum_{k=3}^{n+1}a_{n+1,k}h_k=-b_{n+1,1}+
b_{n+1,2}c_{2,1}+b_{n+1,n}c_{n,1}-c_{n+1,1}+\sum_{k=3}^{n-1}b_{n+1,k}c_{k,1}
\]
it can be written as
\[
a_{n+1,1}-b_{n+1,1}+
b_{n+1,2}c_{2,1}+b_{n+1,n}c_{n,1}-c_{n+1,1}-a_{n+1,2}h_2=K_7
\]
putting on the left side what depends on unknown variables and on
the right side, $K_7$, what depends on the induction hypothesis and
on data, note that
\[
K_7=\sum_{k=3}^{n+1}a_{n+1,k}h_k-\sum_{k=3}^{n-1}b_{n+1,k}c_{k,1}.
\]

Moreover, as $(1,\om_i)$ $i=1,2,3$ are involutions and we are
considering $n$ even, we get
\[
a_{n+1,n}=-na_{2,1}, \qquad b_{n+1,n}=-nb_{2,1}, \qquad
c_{n+1,n}=-nc_{2,1}.
\]
Then, gathering together every equalities above we obtain
\[
\frac{1}{2}(a_{n+1,2}a_{2,1}+a_{n+1,n}a_{n,1})
-\frac{1}{2}(b_{n+1,2}b_{2,1}+b_{n+1,n}b_{n,1})+
b_{n+1,2}c_{2,1}+b_{n+1,n}c_{n,1}-
\]
\[
-\frac{1}{2}(c_{n+1,2}c_{2,1}+c_{n+1,n}c_{n,1})-a_{n+1,2}h_2=K_8
\]
where $K_8=K_7-K_1+K_2+K_3$, o equivalently
\[
\left(\frac{1}{2}a_{2,1}-h_2\right)a_{n+1,2}+\frac{1}{2}a_{n+1,n}a_{n,1}+
\left(c_{2,1}-\frac{1}{2}b_{2,1}\right)b_{n+1,2}-\frac{1}{2}b_{n+1,n}b_{n,1}+
\]
\[
+\left(b_{n+1,n}-\frac{1}{2}c_{n+1,n}\right)c_{n,1}-
\frac{1}{2}c_{n+1,2}c_{2,1}=K_8
\]
and
\[
\left(\frac{1}{2}a_{2,1}-h_2\right)(-2a_{n,1}+K_4)+\frac{1}{2}a_{n+1,n}a_{n,1}+
\left(c_{2,1}-\frac{1}{2}b_{2,1}\right)(-2b_{n,1}+K_5)-\frac{1}{2}b_{n+1,n}b_{n,1}
+\]
\[\left(b_{n+1,n}-\frac{1}{2}c_{n+1,n}\right)c_{n,1}-
\frac{1}{2}(-2c_{n,1}+K_6)c_{2,1}=K_8
\]
so
\[
(2h_2-a_{2,1}+\frac{1}{2}a_{n+1,n})a_{n,1}+(b_{2,1}-2c_{2,1}-\frac{1}{2}b_{n+1,n})b_{n,1}+
\]
\[
+(c_{2,1}+b_{n+1,n}-\frac{1}{2}c_{n+1,n})c_{n,1}=K_9
\]
where
\[
K_9=K_8-\left(\frac{1}{2}a_{2,1}-h_2\right)K_4-\left(c_{2,1}-\frac{1}{2}b_{2,1}\right)K_5+
\frac{1}{2}c_{2,1}K_6.
\]


Finally
\[
\left(2h_2-\left(\frac{n}{2}+1\right)a_{2,1}\right)a_{n,1}+\left(\left(\frac{n}{2}+1\right)b_{2,1}-2c_{2,1}\right)b_{n,1}
+\left(\left(\frac{n}{2}+1\right)c_{2,1}-nb_{2,1}\right)c_{n,1}=K_9.
\]

It is important to note that for $m=1\cdots9$, $K_m$ depends only on
the solution for $\Ri_{n-1}$ and on $h$.

In summary, we have to study the consistency of the linear system
\[
\left\{
  \begin{array}{ll}
    a_{n,1}-b_{n,1}+c_{n,1}=h_{n}+\sum_{k=2}^{n-1}(a_{n,k}h_k-b_{n,k}c_{k,1}) \\
\left(2h_2-\left(\frac{n}{2}+1\right)a_{2,1}\right)a_{n,1}+\left(\left(\frac{n}{2}+1\right)b_{2,1}-2c_{2,1}\right)b_{n,1}
+\left(\left(\frac{n}{2}+1\right)c_{2,1}-nb_{2,1}\right)c_{n,1}=K_9.
  \end{array}
\right.
\]
Consider the matrix
\[
\begin{pmatrix}
1& -1\\
2h_2-\left(\frac{n}{2}+1\right)a_{2,1} &
\left(\frac{n}{2}+1\right)b_{2,1}-2c_{2,1}
\end{pmatrix}
\]
whose determinant is $\left(\frac{n}{2}-1\right)(b_{2,1}-a_{2,1})$
by $(\ref{E:n=2})$. Since we can choose $a_{2,1}\neq b_{2,1}$ in
$(\ref{E:n=2})$ and $n\geq4$ the system is consistent and we get
solutions for $n+1$.

\

To complete the proof of the theorem we are going to substitute in
(\ref{E:ladproductode3}) the solutions found before for
(\ref{E:lahproductode3}). The involutions $(1,\om_i)$ and $(\de_i,
\om_i)$ share the $A$-sequence denoted before by $A^{\om_i}$. In
particular, we have to recall that $A_1^{\om_1}=-a_{2,1}$,
$A_1^{\om_2}=-b_{2,1}$ and $a_{2,1}\neq b_{2,1}$ where
$(1,\om_1)=(a_{i,j})$ and $(1,\om_2)=(b_{i,j})$. Additionally, we
will prove also that we can always suppose that $\de_3\equiv1$ in
(\ref{E:ladproductode3}). Doing this (\ref{E:ladproductode3}) turns
into any of the following three equivalent equations
\[
\de_1(x)\de_2(\om_1(x))=d(x)
\]
\[
(\de_1,\om_1)(\de_2,\om_2)(1)=d(x)
\]
or
\begin{equation}\label{E:ladreducida}
(\de_1,\om_1)d(x)=(\de_2,\om_2)(1)
\end{equation}
Suppose that
\[
(\de_1,\om_1)=(u_{i,j})_{i,j\in\N}, \quad
(\de_2,\om_2)=(v_{i,j})_{i,j\in\N}, \quad
d(x)=\sum_{k\geq0}d_{k,0}x^k.
\]
So, in $\Ri_0$ equation (\ref{E:ladreducida}) holds tautologically.
In $\Ri_1$ the equation (\ref{E:ladreducida}) is
\[
\left(
  \begin{array}{cc}
    1 & 0 \\
    u_{1,0} & -1 \\
  \end{array}
\right)\left(
         \begin{array}{c}
           1 \\
           d_{1,0} \\
         \end{array}
       \right)=\left(
  \begin{array}{cc}
    1 & 0 \\
    v_{1,0} & -1 \\
  \end{array}
\right)\left(
         \begin{array}{c}
           1 \\
           0 \\
         \end{array}
       \right)
\]
and the linear equation
\begin{equation}\label{E:lau10v10}
u_{1,0}-v_{1,0}=d_{1,0}
\end{equation}
has infinite solutions for any $d_{1,0}$ for the unknowns $u_{1,0}$
and $v_{1,0}$ and then we have solutions in $\Ri_1$. In $\Ri_2$  the
equation (\ref{E:ladreducida}) is
\[
\left(
  \begin{array}{ccc}
    1 & 0 & 0\\
    u_{1,0} & -1 &0 \\
    -\frac{1}{2}u_{1,0}u_{2,1} & u_{2,1} & 1
  \end{array}
\right)\left(
         \begin{array}{c}
           1 \\
           d_{1,0} \\
           d_{2,0}
         \end{array}
       \right)=\left(
  \begin{array}{ccc}
    1 & 0 & 0\\
    v_{1,0} & -1 & 0\\
     -\frac{1}{2}v_{1,0}v_{2,1} & v_{2,1} & 1
  \end{array}
\right)\left(
         \begin{array}{c}
           1 \\
           0 \\
           0
         \end{array}
       \right)
\]
then we have to solve the system whose first equation is
$(\ref{E:lau10v10})$ and the second one is
\[
-\frac{1}{2}u_{1,0}u_{2,1}+u_{2,1}d_{1,0}+d_{2,0}=-\frac{1}{2}v_{1,0}v_{2,1}
\]
or equivalently
\[
\left(d_{1,0}-\frac{1}{2}u_{1,0}\right)u_{2,1}+d_{2,0}=-\frac{1}{2}v_{1,0}v_{2,1}.
\]
Using $(\ref{E:lau10v10})$ we get
\[
\left(v_{1,0}-\frac{1}{2}u_{1,0}\right)u_{2,1}-\frac{1}{2}v_{1,0}v_{2,1}=d_{2,0}.
\]
By the construction of Riordan matrices by means of the A-sequence
we obtain
\[
u_{2,1}=a_{2,1}-u_{1,0},\qquad v_{2,1}=b_{2,1}-v_{1,0},
\]
consequently
\[
\left(v_{1,0}-\frac{1}{2}u_{1,0}\right)(a_{2,1}-u_{1,0})-\frac{1}{2}v_{1,0}(b_{2,1}-v_{1,0})=d_{2,0}
\]
or
\[
\frac{1}{2}(u_{1,0}-v_{1,0})^2+\left(v_{1,0}-\frac{1}{2}u_{1,0}\right)a_{2,1}-\frac{1}{2}v_{1,0}b_{2,1}=d_{2,0}.
\]
Doing some computations we obtain
\[
-a_{2,1}u_{1,0}+(2a_{2,1}-b_{2,1})v_{1,0}=2d_{2,0}-d_{1,0}^2
\]
So, the linear system to solve is
\[
\left\{
  \begin{array}{ll}
    u_{1,0}-v_{1,0}=d_{1,0}, \\
   -a_{2,1}u_{1,0}+(2a_{2,1}-b_{2,1})v_{1,0}=2d_{2,0}-d_{1,0}^2 \hbox{.}
  \end{array}
\right.
\]
Using (\ref{E:n=2}) and (\ref{E:lau10v10}) we prove that the above
system has solutions. Consequently we solve our problem for $\Ri_2$.
%

We proceed by induction in a similar way to the previous case.
Suppose we have solved (\ref{E:ladproductode3}) in $\Ri_n$ and we
want to solve it in $\Ri_{n+1}$. Then we have the linear system with
the $n$ previous equations and the new equation
\begin{equation}{\label{E:newEqu}}
\sum_{j=0}^{n+1}u_{n+1,j}d_{j,0}=v_{n+1,0}
\end{equation}
which is the same as
\[
 u_{n+1,0}-v_{n+1,0}=d_{n+1,0}-\sum_{j=1}^{n}u_{n+1,j}d_{j,0}
\]
or
\[
u_{n+1,0}-v_{n+1,0}=L_{1}
\]
with $L_{1}=d_{n+1,0}-\sum_{j=1}^{n}u_{n+1,j}d_{j,0}$. As in the
case of the previous symbols $K_m$, $L_j$ groups together terms
depending on induction hypothesis and data.

Since the elements $(2k+1,0)$ in an involution are arbitrary and the
elements $(2k,0)$ are given by (\ref{E:d(n,0) n par}) we have to
distinguish two cases. In the case $n+1$ odd, $u_{n+1,0}$ and
$v_{n+1,0}$ are arbitrary. Then, the linear system with the new
equation has solutions. In the case that $n+1$ even, these elements
are given by (\ref{E:d(n,0) n par}). Then we must study the
consistency of the linear system
\[
\left\{
  \begin{array}{ll}
    u_{n,0}-v_{n,0}=d_{n,0}-\sum_{j=1}^{n-1}u_{n,j}d_{j,0},  \\
    \sum_{j=0}^{n+1}u_{n+1,j}d_{j,0}=v_{n+1,0} \hbox{.}
  \end{array}
\right.
\]
By using formula (\ref{E:d(n,0) n par}) we get
\[
u_{n+1,0}=-\frac{1}{2}(u_{n+1,1}u_{1,0}+u_{n+1,n}u_{n,0})+ L_2
\]
\[
v_{n+1,0}=-\frac{1}{2}(v_{n+1,1}v_{1,0}+v_{n+1,n}v_{n,0})+ L_3
\]
where
\[
L_2=-\frac{1}{2}\sum_{k=2}^{n-1}u_{n+1,k}u_{k,0}\quad
\text{and}\quad L_3=-\frac{1}{2}\sum_{k=2}^{n-1}v_{n+1,k}v_{k,0}.
\]
Moreover, by means of the horizontal construction of a Riordan
matrix we have
\[
u_{n+1,1}=-u_{n,0}+L_4, \qquad v_{n+1,1}=-v_{n,0}+L_5,
\]
where
\[
L_4=\sum_{k=1}^nA_k^{\om_1}u_{n,k}, \qquad
L_5=\sum_{k=1}^nA_k^{\om_2}v_{n,k}
\]
Beside, we know that
\[
u_{n+1,n}=-u_{1,0}+na_{2,1}\qquad \text{and}\qquad
v_{n+1,n}=-v_{1,0}+nb_{2,1}.
\]
By $(\ref{E:lau10v10})$ we can write equation $(\ref{E:newEqu})$ as
\[
u_{n+1,0}+u_{n+1,1}(u_{1,0}-v_{1,0})-v_{n+1,0}=L_6,
\]
where $L_6=-\sum_{j=2}^{n+1}u_{n+1,j}d_{j,0}$. Now we replace the
expressions above to obtain
\[
-\frac{1}{2}(u_{n+1,1}u_{1,0}+u_{n+1,n}u_{n,0})+
u_{n+1,1}(u_{1,0}-v_{1,0})-\frac{1}{2}(v_{n+1,1}v_{1,0}+v_{n+1,n}v_{n,0}))=L_6-L_2+L_3
\]
\[
(\frac{1}{2}u_{1,0}-v_{1,0})u_{n+1,1}-\frac{1}{2}u_{n+1,n}u_{n,0}-\frac{1}{2}(v_{n+1,1}v_{1,0}+v_{n+1,n}v_{n,0}))=L_6-L_2+L_3
\]
\[
(\frac{1}{2}u_{1,0}-v_{1,0})(-u_{n,0}+L_4)-\frac{1}{2}(-u_{1,0}+na_{2,1})u_{n,0}-
\frac{1}{2}((-v_{n,0}+L_5)v_{1,0}+(-v_{1,0}+nb_{2,1})v_{n,0}))=L_6-L_2+L_3
\]
\[
(\frac{1}{2}u_{1,0}-v_{1,0})(-u_{n,0})-\frac{1}{2}(-u_{1,0}+na_{2,1})u_{n,0}-
\frac{1}{2}((-v_{n,0})v_{1,0}+(-v_{1,0}+nb_{2,1})v_{n,0}))=L_7
\]
where
\[
L_7=L_6-L_2+L_3-(\frac{1}{2}u_{1,0}-v_{1,0})L_4+\frac{1}{2}v_{1,0}L_5
\]
reorganizing the variables we get the linear system
\[
\left\{
  \begin{array}{ll}
    u_{n,0}-v_{n,0}=d_{n,0}-\sum_{j=1}^{n-1}u_{n,j}d_{j,0},  \\
    \left(v_{1,0}-\frac{n}{2}a_{2,1}\right)u_{n,0}+
\left(\frac{n}{2}b_{2,1}-v_{1,0}\right)v_{n,0}=L_7 \hbox{.}
  \end{array}
\right.
\]
It has solutions because
\[
\begin{vmatrix}
1 & -1 \\
v_{1,0}-\frac{n}{2}a_{2,1} & \frac{n}{2}b_{2,1}-v_{1,0}
\end{vmatrix}=\frac{n}{2}b_{2,1}-v_{1,0}+v_{1,0}-\frac{n}{2}a_{2,1}=\frac{n}{2}(b_{2,1}-a_{2,1})\neq0
\]
that is the needed condition for the equation
(\ref{E:lahproductode3}) holds. Then the case $n+1$ has also
solutions.

So, we have proved that our result is true in $\Ri_n$ for every
$n\in\N$. Since the group $\mathcal{R}$ is the inverse limit of the
inverse sequence of groups
$\{(\mathcal{R}_n)_{n\in\N},(P_n)_{n\in\N}\}$ and $P_n$ transforms
solutions in $\Ri_{n+1}$ to solutions in $\Ri_n$, the proof is
finished.
\end{proof}

Note that, if $D\in<\mathfrak{I}>$, the main diagonal of $D$ is the
same as the main diagonal of one of the diagonal Riordan
involutions. In each of the cases, we can multiply $D$ by one of the
diagonal involutions to get a Riordan matrix $(d,h)$ with $d_0=1$
and $h_1=-1$, then by Theorem \ref{T:3invo} we obtain that $D$ is a
product of at most four involutions. Then we have proved Theorem
\ref{T:main} for $\Ri$ and $\Ri_n$, $n\geq1$.

In fact we can describe, up to isomorphism, the group generated by
involutions using the commutator subgroup and the semidirect product
concept in the following way.

\begin{thm}\label{C:prodsemi}
\[<\mathfrak{I}> \thickapprox [\Ri, \Ri]_0 \rtimes \mathcal{K}\]
where $[\Ri, \Ri]_0=\Omega_0\cap[\Ri, \Ri]$ and $\mathcal{K}=\{I,
-I, \mathcal{I}^{+}_0, \mathcal{I}^{-}_0 \}$.
\end{thm}

\begin{proof}
Note that $<\mathfrak{I}> \subseteq [\Ri, \Ri]_0\mathcal{K}$ where
$[\Ri, \Ri]_0\mathcal{K}$ represents the set of Riordan matrices
obtained by multiplying an element in $[\Ri, \Ri]_0$ with an element
in $\mathcal{K}$ in such order. Now suppose $DK\in[\Ri,
\Ri]_0\mathcal{K}$ with $D\in[\Ri, \Ri]_0$ and $K\in\mathcal{K}$.
Then
$D\mathcal{I}^{+}_0=\mathcal{I}^{+}_{\al}\mathcal{I}^{+}_{\be}\mathcal{I}^{+}_{\g}$
by Theorem \ref{T:3invo}. Consequently
$D=\mathcal{I}^{+}_{\al}\mathcal{I}^{+}_{\be}\mathcal{I}^{+}_{\g}\mathcal{I}^{+}_0$,
hence $DK\in<\mathfrak{I}>$. So, $<\mathfrak{I}> = [\Ri,
\Ri]_0\mathcal{K}$. We also proved in Theorem \ref{T:3invo} that, in
fact, $[\Ri, \Ri]_0= [\Ri, \Ri]\cap<\mathfrak{I}>$. Therefore $[\Ri,
\Ri]_0\unlhd <\mathfrak{I}>$. Finally, $[\Ri,
\Ri]_0\cap\mathcal{K}=\{I\}$. This implies the result by using
\cite{amstrong} page 133.
\end{proof}

{\bf Final remark:} It is easy to prove that any element in the
group generated by involutions in $\Ri_1$ can be described as the
product of two of them. Using Theorem \ref{T:invo} herein and
Corollary 7 in \cite{formula}, the matrix
\[
\left(
  \begin{array}{ccc}
    1 & 0 & 0 \\
    0 & 1 & 0 \\
    1 & 0 & 1 \\
  \end{array}
\right)=\left(
  \begin{array}{ccc}
    1 & 0 & 0 \\
    1 & -1 & 0 \\
    -1 & 2 & 1 \\
  \end{array}
\right)\left(
  \begin{array}{ccc}
    1 & 0 & 0 \\
    1 & -1 & 0 \\
    0 & 0 & 1 \\
  \end{array}
\right)\left(
  \begin{array}{ccc}
    1 & 0 & 0 \\
    0 & -1 & 0 \\
    0 & -2 & 1 \\
  \end{array}
\right)\left(
  \begin{array}{ccc}
    1 & 0 & 0 \\
    0 & -1 & 0 \\
    0 & 0 & 1 \\
  \end{array}
\right)
\]
points out that, for $n\geq2$, in $\Ri_n$ and in $\Ri$ there are
elements in the group generated by involutions that can not be
described as the product of three or less involutions.

{\bf Acknowledgment:} The first and second authors were partially
supported by grant MINECO, MTM2015-63612-P.

\end{document}